\documentclass[12pt]{article}
\usepackage[english]{babel}
\usepackage{amsthm,amsfonts, amsbsy, amssymb,amsmath,graphicx}
\usepackage{graphics}
\usepackage[all]{xy}

\newtheorem{theorem}{Theorem}
\newtheorem{lemma}{Lemma}
\newtheorem{corollary}{Corollary}

\newtheorem{proposition}{Proposition}
\theoremstyle{definition}
\newtheorem{definition}{Definition}
\newtheorem{example}{Example}
\theoremstyle{remark}
\newtheorem{remark}{Remark}

\def\Z{{\mathbb Z}}

\def\R{{\mathbb R}}

\date{}

\author{Igor Nikonov}
\title{On universal parity on free two-dimensional knots}
\date{}

\begin{document}

\maketitle

\begin{abstract}
We prove that the Gaussian parity on free two-dimensional knots is universal.
\end{abstract}

\section*{Introduction}

A parity is a rule to assign labels $0$ and $1$ to the crossings of knot diagrams in a way compatible with
Reidemeister moves (Fig.~\ref{fig:parity_axioms}). It was introduced by V.O. Manturov~\cite{M3} in his researches concerning free knots.

\begin{figure}[h]
\centering\includegraphics[width=0.7\textwidth]{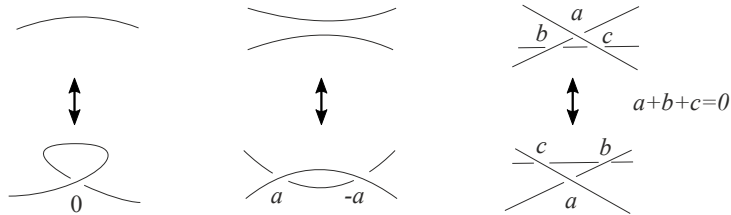}
\caption{Parity axioms}\label{fig:parity_axioms}
\end{figure}

An example of such labeling gives the {\em link parity} which is defined on diagrams of $2$-component links. A crossing is even in link parity if the both arcs belong to one component (of a two-component link), and is odd if the arcs belong to different components (Fig.~\ref{fig:gaussian_link_parity} right).

\begin{figure}[h]
\centering\includegraphics[height=0.2\textwidth]{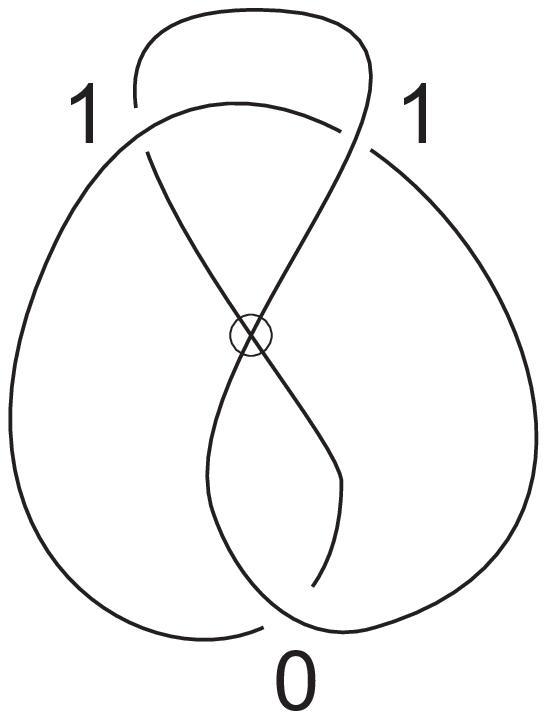}\qquad \includegraphics[height=0.2\textwidth]{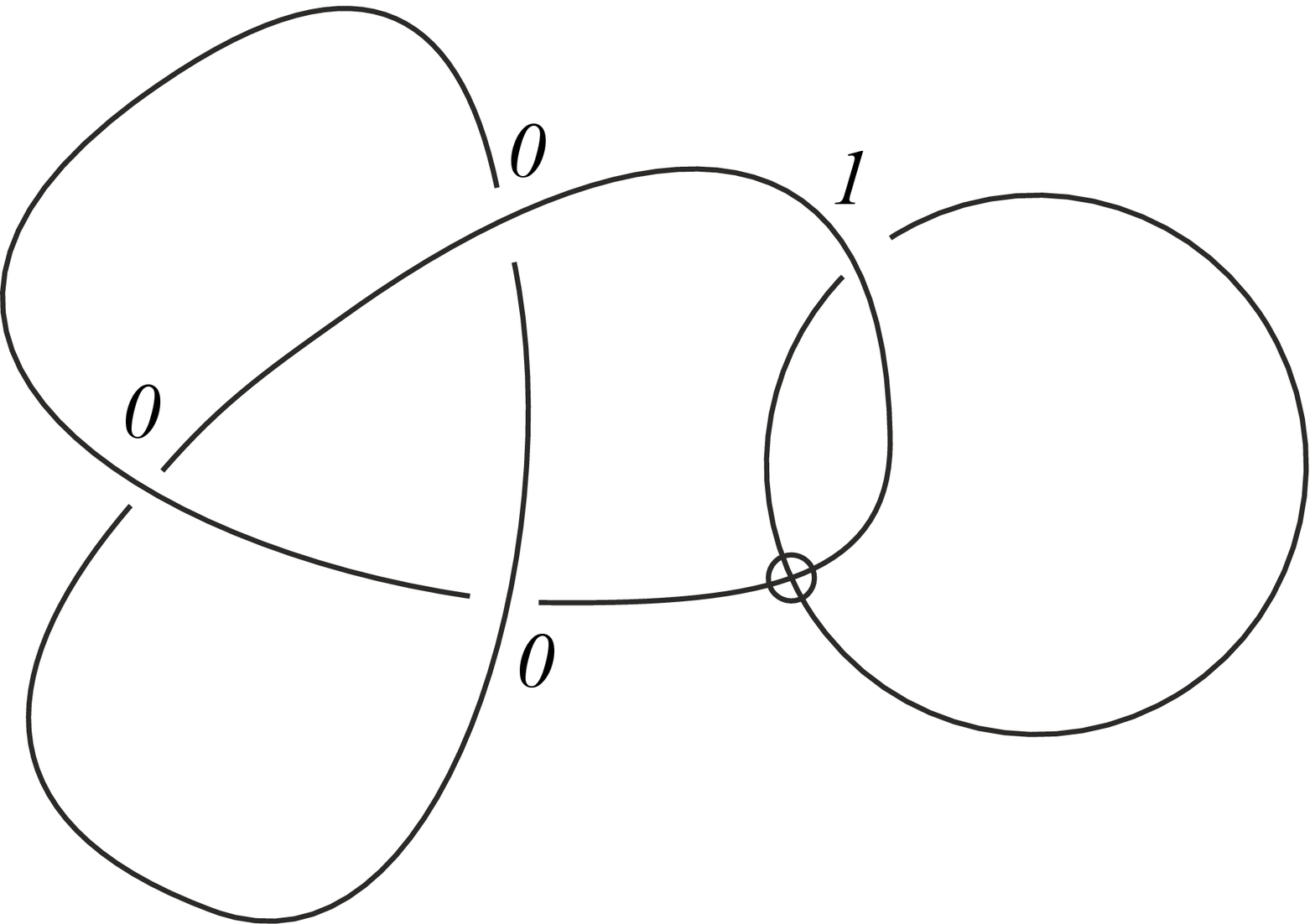}
\caption{Gaussian (left) and link (right) parities}\label{fig:gaussian_link_parity}
\end{figure}

Another example is {\em Gaussian parity} defined on diagrams of virtual knots. The Gaussian parity of a crossing of a virtual knot diagram is the parity of the number of (classical) crossings that one encounters while moving along the knot from the chosen crossing back to it (Fig.~\ref{fig:gaussian_link_parity} left).

The notion of parity has found various applications in knot
theory. It allows to strengthen knot invariants, to prove minimality
theorem and to construct (counter)examples~\cite{IMN, IMN1, M1,M2,M3,M4,M5,M6,M7}.

In~\cite{FM} D.A. Fedoseev and V.O. Manturov generalized the notion of parity to $2$-knots and showed how parity can be used to refine $2$-knot invariants.

The aim of the present paper is to find how much information can be extracted from parities. We show that any parity on diagrams of free $2$-knots (i.e. $2$-knots considered up to crossing switch and virtualization) is reduced to a concrete $\Z_2$-index --- the Gaussian parity.

The paper is organized as follows. In the next section we remind the definition of $2$-knots, Gauss diagrams of $2$-knots and free $2$-knots. In Section~\ref{sec:parity} we give the definition of a parity on $2$-knots with coefficients in an abelian group. In Section~\ref{sec:main_result} we prove that the Gaussian parity is universal for free $2$-knots. In the last section we introduce oriented parities.

\section{Parities of $2$-knots}\label{sec:2-knots}

Let us remind the definition of classical and abstract $2$-knots~\cite{Ros1,Ros2,Winter}.

\begin{definition}
A {\em $2$-knot} is a smooth embedding $S^2\hookrightarrow\R^4$ considered up to isotopy.
\end{definition}

\begin{example}
Spun construction is one of methods to construct $2$-knots. Given a knotted curve in $\R^3$ with ends on a fixed axis, a $2$-knot is the result of rotation of the curve around the axis which is considered as an axis in $\R^4\supset\R^3$.

\begin{figure}[h]
\centering\includegraphics[width=0.15\textwidth]{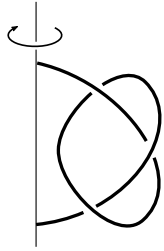}
\caption{Spun trefoil}
\end{figure}
\end{example}

Like the usual knots, $2$-knots can be presented by diagrams

\begin{definition}
A {\em diagram} $D$ of a $2$-knot is a projection in general position of the $2$-knot
$$\xymatrix{S^2 \ar[dr]_D \ar@{^{(}->}[r] & \R^4 \ar[d]  \\ & \R^3}$$
The vertical arrow is a linear projection from $\R^4$ to $\R^3$.
\end{definition}

Locally, the points of a $2$-diagram $D$ belongs to one of the four types: regular point, double point, triple point and pinch point (Fig.~\ref{fig:singularities}).
\begin{figure}[h]
\centering\includegraphics[width=0.6\textwidth]{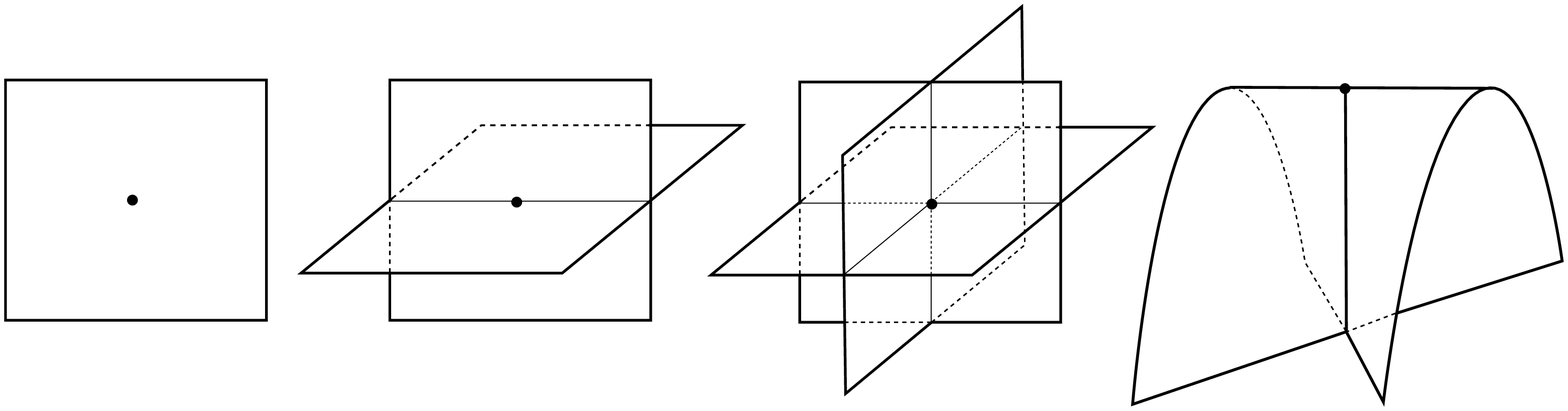}
\caption{Types of diagram points}\label{fig:singularities}
\end{figure}

Like crossings of usual knots, double lines of $2$-knot diagrams have under/overcrossing structure and oriented. Undersheets are depicted with a gap at the double line. The orientation is induced by the orientations of the embedded sphere and $\R^3$ (Fig.~\ref{fig:double_line_orientation}): the normal to the overcrossing sheet $\vec{n}_+$, the normal to the under-crossing sheet $\vec{n}_-$ and the tangent vector to the double line $\vec{v}$ must form a positive basis of $\R^3$.
\begin{figure}[h]
\centering\includegraphics[width=0.3\textwidth]{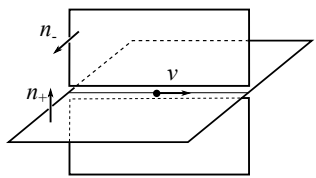}
\caption{Orientation of double lines: the triple $(\vec{n}_+,\vec{n}_-,\vec{v})$ have the positive orientation in $\R^3$}\label{fig:double_line_orientation}
\end{figure}

Reidemester moves on diagrams have their counterparts for $2$-knots.

\begin{theorem}[Roseman~\cite{Ros1}]
Two diagrams $D_1$ and $D_2$ correspond to isotopic $2$-knots if and only if they can be connected by a finite sequence of isotopies and moves $\mathcal{R}_1$--$\mathcal{R}_7$ (Fig.~\ref{fig:roseman_move}).
\begin{figure}[h]
\centering\includegraphics[width=0.8\textwidth]{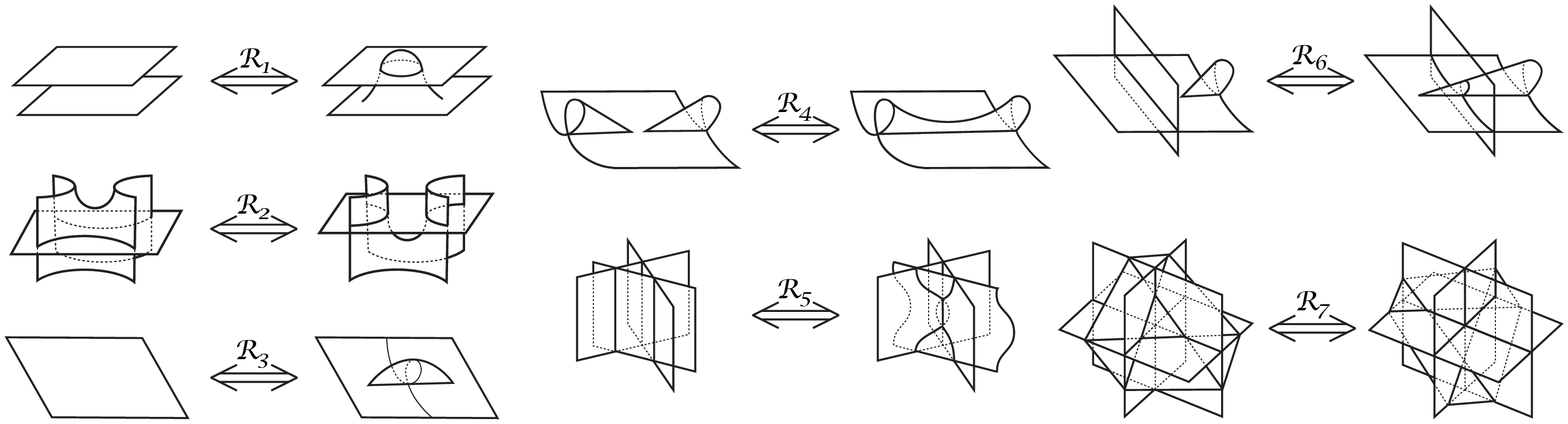}
\caption{Roseman moves}\label{fig:roseman_move}
\end{figure}
\end{theorem}

Gauss diagrams present an alternative way to define $1$-knots. Given a knot diagram, its Gauss diagram shows paired pre-images of double points of the diagram with orientation indicated (Fig.~\ref{fig:gauss_diagram}). Gauss diagrams encode virtual knots rather than classical knots.
\begin{figure}[h]
\centering\includegraphics[height=0.15\textwidth]{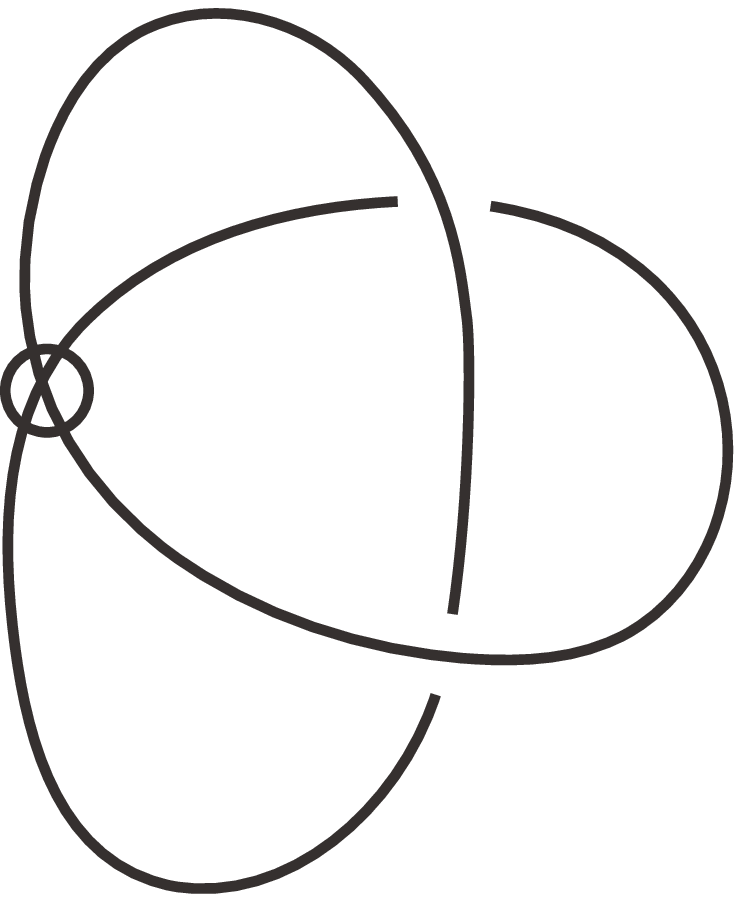}\qquad\includegraphics[height=0.15\textwidth]{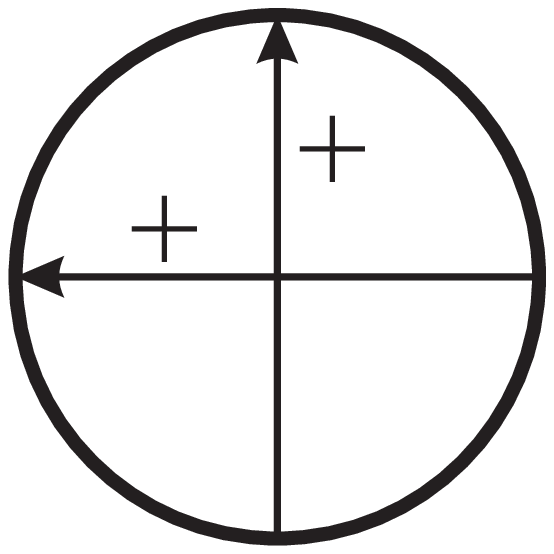}
\caption{Virtual trefoil and its Gauss diagram}\label{fig:gauss_diagram}
\end{figure}

Informally, Gauss diagram of a diagram $D$ of some $2$-knot is the pre-image of singular point of the diagram, equipped with orientations and  signs. This pre-image is the closure of the curves which maps to double points of $D$ by the mapping $S^2\to D$. The general properties of these curves lead to the following definition~\cite{FM,Winter}.

\begin{definition}
A {\em $2$-Gauss diagram} is a set $\Gamma$ of curves in the sphere $S^2$. The curves are oriented, marked with signs $\pm$ and obeying the following conditions:
\begin{enumerate}
\item Each curve is closed or terminates in a cusp.
\item  Each curve in $\Gamma$ is paired with a unique second curve in $\Gamma$, and one curve in
the pair is marked as {\em over} (+), the other is marked as {\em under} (-).
\item  Two curves that terminate in a single cusp are paired together, with the orientations
both pointing toward or both pointing away from the cusp.
\item  If two curves cross, then the curves they are paired with cross two curves which are paired (that is,
triple points appear three times on the sphere).
\end{enumerate}
\end{definition}

\begin{remark}\label{rem:geometrical_triple_point}
When a Gauss diagram appears as the preimage of double lines of a $2$-knot diagram, the signs and orientations of curves at any triple point are compatible in the following way:
\begin{itemize}
\item The marks of the crossing curves in the sheets are $\{+,+\}$, $\{+,-\}$, $\{-,-\}$.
\item The orientations of the pairs $(c_+,b_+)$, $(c_-,a_+)$ and $(b_-,a_-)$ coincide (see Fig.~\ref{fig:geometrical_triple_point}).
\end{itemize}

\begin{figure}[h]
\centering\includegraphics[width=0.7\textwidth]{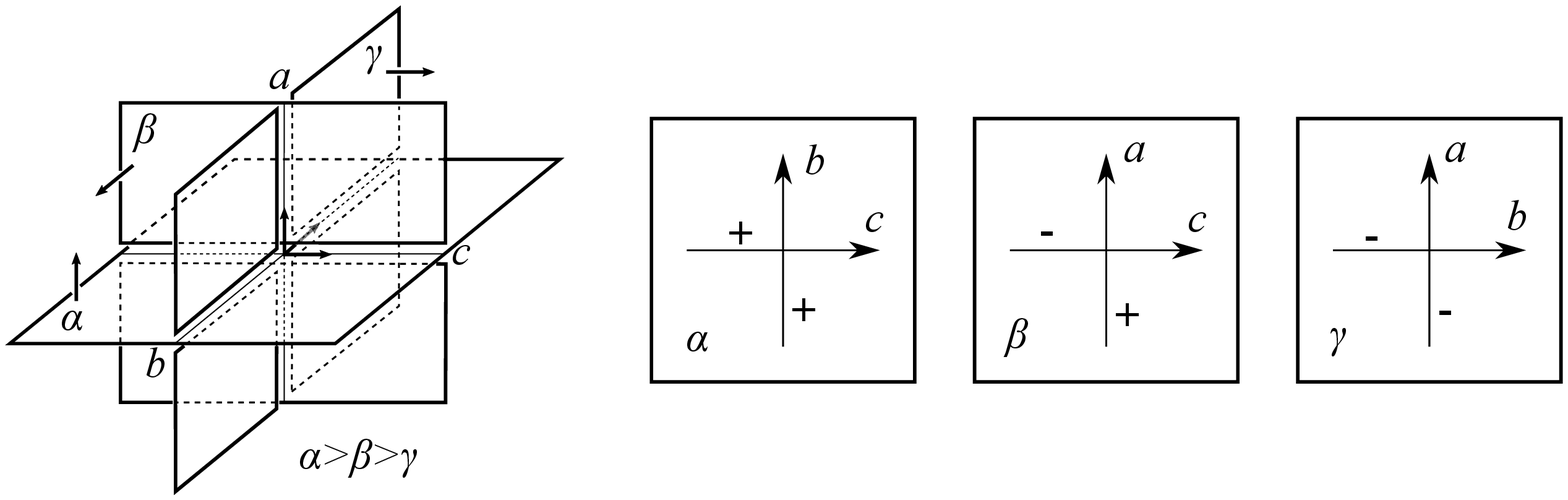}
\caption{Orientations at a geometrical triple point}\label{fig:geometrical_triple_point}
\end{figure}

We call such triple points {\em geometrical}.
\end{remark}

Let $\Gamma$ be a $2$-Gauss diagram. By gluing together the paired curves of $\Gamma$, one gets a factor space $D$ of the sphere $S^2$ called an {\em abstract $2$-knot diagram}. If the Gauss diagram comes from a diagram of a $2$-knot then $D$ is homeomorphic to this diagram. In general case, $D$ is a two dimensional complex for which there may be no embeddings into $\R^3$. But the definition ensures that locally $D$ looks like complexes in Fig.~\ref{fig:singularities}. Paired curves $a_+$ and $a_-$ of $\Gamma$ correspond to a line $a$ of double points in $D$.

Below we will use the abstract $2$-knot diagrams $D$ as counterparts of the Gauss diagrams $\Gamma$.

The Roseman moves on $2$-knot diagrams induce transformations of the $2$-Gauss diagrams (Fig.~\ref{fig:gauss_roseman_moves}).
\begin{figure}[h]
\centering\includegraphics[width=\textwidth]{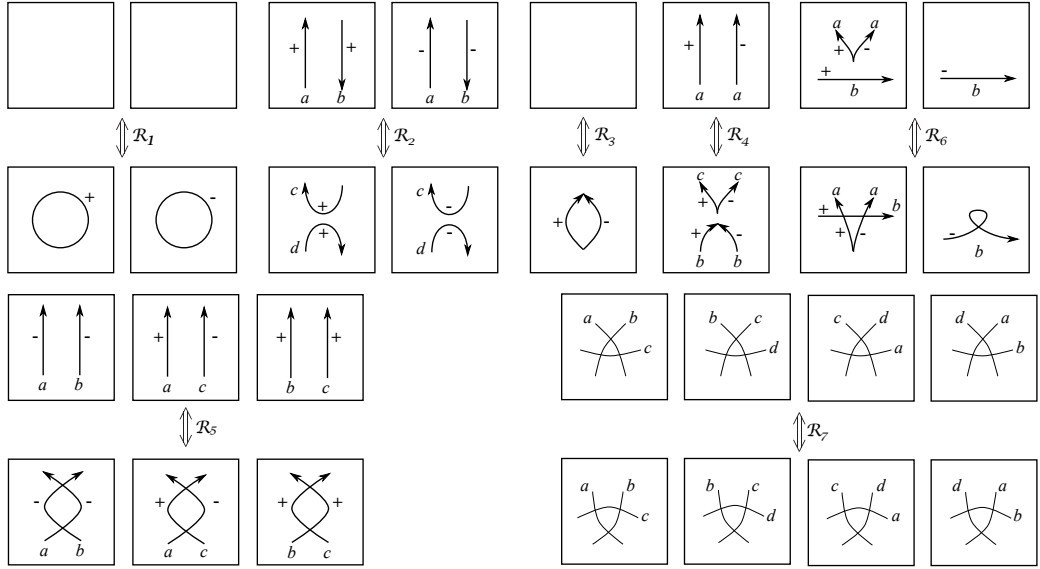}
\caption{Roseman moves on Gauss diagrams}\label{fig:gauss_roseman_moves}
\end{figure}

\begin{definition}
{\em Virtual $2$-knots} are equivalence classes of Gauss diagrams on the sphere modulo Roseman moves.
\end{definition}

We consider two operations on $2$-Gauss diagrams. A {\em crossing switch} interchanges the over and the undersheets at a double line. A {\em virtualization} changes the orientations of two paired curves in the diagram (Fig.~\ref{fig:switch_virtualization}).

\begin{figure}[h]
\centering\includegraphics[height=4cm]{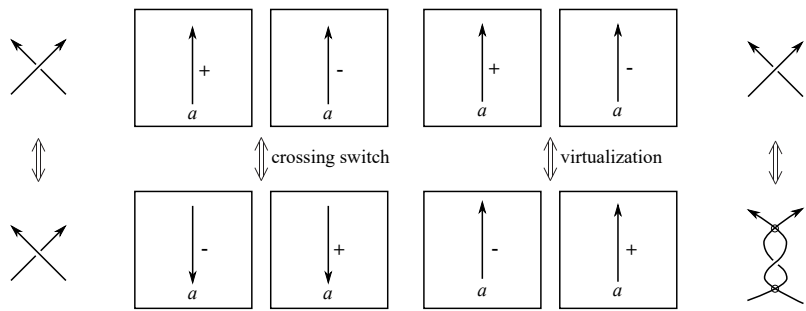}
\caption{Crossing switch and virtualization}\label{fig:switch_virtualization}
\end{figure}

Crossing switches and virtualizations can turn geometrical triple points into non geometrical ones.

\begin{definition}
{\em Flat $2$-knots} are virtual $2$-knots modulo crossing switches. {\em Free $2$-knots} are virtual $2$-knots modulo crossing switches and virtualizations.
\end{definition}

\section{Parity on $2$-knots}\label{sec:parity}

Let us define the main notion of the paper. Our definition combines the definition of a parity for $2$-knots given in~\cite{FM} and the definition of a parity with coefficients of ~\cite{IMN1}.

\begin{definition}\label{def:parity}

Let $\mathcal K$ be a virtual $2$-knot. A {\em parity} with coefficients in an abelian group $A$ on the diagrams of a $2$-knot $\mathcal K$ is a family of maps $p_D\colon \mathcal D(D)\to A$, where $\mathcal D(D)$ is the set of double lines of an abstract diagram $D$ of the knot $\mathcal K$, which obeys the following properties.
\begin{enumerate}
\item ({\em local correspondence property}) let diagrams $D$ and $D'$ be connected by a Roseman move. If double lines $a\in\mathcal D(D)$ and $a'\in\mathcal D(D')$ contain arcs which do not affected by the move and correspond to each other then $p_D(a)=p_{D'}(a')$.
\item ({\em triple point property}) if double lines $a,b,c\in\mathcal D(D)$ intersect in a triple point of the diagram $D$ then $p_D(a)+p_D(b)+p_D(c)=0$.
\item ({\em pinch property}) the parity the double line $a$ of any pinch is zero: $p_D(a)=0$. A {\em pinch} is the pair of two curves which appears after a Roseman move $\mathcal R_3$ (Fig.~\ref{fig:pinch}).
\end{enumerate}

\begin{figure}[h]
\centering\includegraphics[height=2cm]{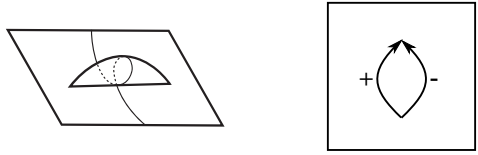}
\caption{The abstract diagram and the Gauss diagram of a pinch}\label{fig:pinch}
\end{figure}
\end{definition}

If $\mathcal K$ is presented by a Gauss diagram $\Gamma$ the parity is assigned to the curves in $\Gamma$, and for any two paired curves $a_+$, $a_-$ they have the same parity.

\begin{example}[\cite{FM}]
The {\em Gaussian parity} $p^G$ of a double line $a$ of a $2$-knot diagram is the parity of the number of intersections with double lines of a path connecting the corresponding points on the paired curves $a$. The double line and the tangent vector to the path must define the same orientation of the sphere at the end points.

\begin{figure}[h]
\centering\includegraphics[height=2cm]{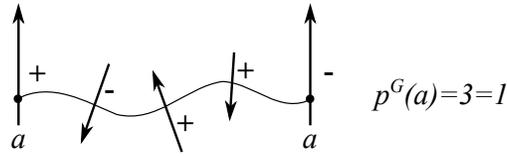}
\caption{Definition of Gaussian parity for $2$-knots}\label{fig:gaussian_2parity}
\end{figure}
\end{example}

\begin{example}[\cite{FM}]
The {\em link parity} of a double line of a two-component $2$-link is even if the intersecting leaves belong to one component, and odd if the intersecting leaves belong to different component.
\end{example}

Let us list some general properties of parities on $2$-knots.

\begin{proposition}\label{prop:cusp_parity}
If a double line ends with a cusp then its parity is zero.
\end{proposition}

\begin{proof}
  Let a double line $a$ end in a cusp. Apply a Roseman move $\mathcal R_4$ near the cusp (Fig.~\ref{fig:cusp_parity}). By the local correspondence property, $p_D(a)=p_{D'}(a')=p_{D'}(a'')$. By the pinch property, $p_{D'}(a')=0$. Hence, $p_D(a)=0$.

\begin{figure}[h]
\centering\includegraphics[width=0.6\textwidth]{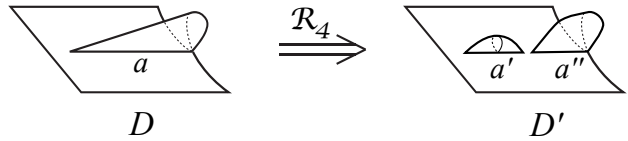}
\caption{Proof of Proposition~\ref{prop:cusp_parity}}\label{fig:cusp_parity}
\end{figure}

\end{proof}

\begin{proposition}\label{prop:R2_parity}
The double line involved in a $\mathcal R_2$ move have the same parities.
\end{proposition}

\begin{proof}
Let $p$ be a parity on diagrams of some $2$-knot. Let $D\to D'$ be a Roseman move $\mathcal R_2$. The move transforms two double lines $a$ and $b$ in $D$ to two double lines $c$ and $d$ in $D'$. Let us choose two sites on the double curve  $a$ as shown in Fig.~\ref{prop:R2_parity}. By the local correspondence property, we have $p_D(a)=p_{D'}(c)$ and $p_D(a)=p_{D'}(d)$. Analogously, $p_D(b)=p_{D'}(c)$ and $p_D(b)=p_{D'}(d)$. Hence, $p_D(a)=p_{D}(b)$.

\begin{figure}[h]
\centering\includegraphics[width=0.6\textwidth]{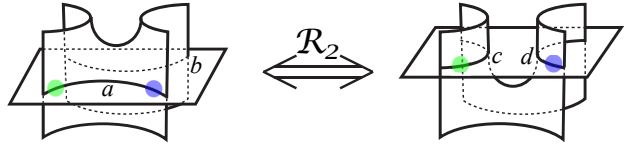}
\caption{Proof of Proposition~\ref{prop:R2_parity}}\label{fig:R2_parity}
\end{figure}
\end{proof}

\begin{proposition}\label{prop:R5_parity}
The sum of parities of double lines which locally form a trigonal prism is equal to $0$.
\end{proposition}

\begin{proof}
Let $a$, $b$, $c$ be the double lines which locally form a trigonal prism. Let $D\to D'$ be the a Roseman move $\mathcal R_5$ applied to the double lines (Fig.~\ref{fig:R5_parity}). Then $a$, $b$, $c$ intersects in a triple point in $D'$. By the triple point property (and the correspondence property), $p_D(a)+p_D(b)+p_D(c)=p_{D'}(a)+p_{D'}(b)+p_{D'}(c)=0$.

\begin{figure}[h]
\centering\includegraphics[width=0.6\textwidth]{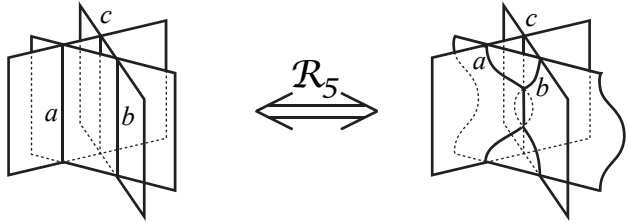}
\caption{Proof of Proposition~\ref{prop:R5_parity}}\label{fig:R5_parity}
\end{figure}
\end{proof}

\begin{remark}
The proofs of various properties of parities for $1$-dimensional knots developed in~\cite{IMN1} can be adapted for $2$-knots. When a $2$-knot diagram looks locally like a prism over a $1$-dimensional diagram, we can imitate Reidemeister moves $R_1, R_2, R_3$ applied to the $1$-dimensional diagram, with the Roseman moves $\mathcal R_3$, $\mathcal R_1$ and $\mathcal R_5$.

\begin{figure}[h]
\centering\includegraphics[width=0.8\textwidth]{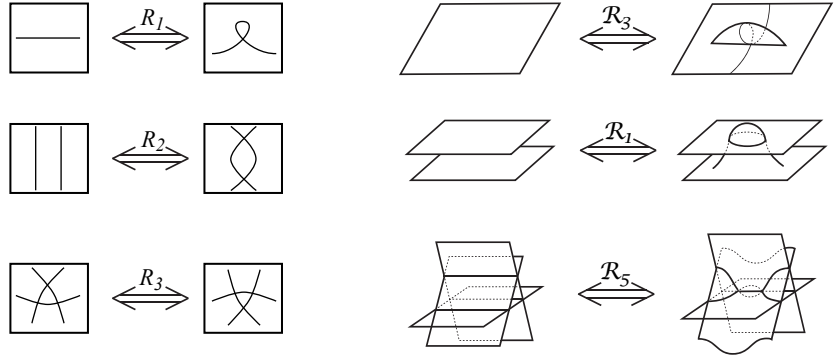}
\caption{Reduction to $1$-dimensional case. For Roseman moves, the $1$-dimensional diagrams appear as the sections by a vertical plane which goes in the middle.}\label{fig:1_dim_reduction}
\end{figure}
\end{remark}

\begin{proposition}\label{prop:parity_mod2}
\begin{enumerate}
\item If double lines $a$ and $b$ form a bigonal prism in a diagram $D$ then $p_D(a)+p_D(b)=0$.
\item For any parity $p$ and any  double line $a\in\mathcal D(D)$ we have $2p_D(a)=0$.
\end{enumerate}
\end{proposition}

\begin{proof}
1. Let $a$ and $b$ form a bigonal prism in a diagram $D$. The section of the prism by a transversal plane looks like a bigon (Fig.~\ref{fig:r2_parity}). Apply a Roseman move $\mathcal R_3$. In the section, it corresponds to a first Reidemeister move. The double lines $a$, $b$ and $c$ form a trigonal prism, hence, $p_{D'}(a)+p_{D'}(b)+p_{D'}(c)=0$. By the pinch property, $p_{D'}(c)=0$. Then $p_{D'}(a)+p_{D'}(b)=0$, and $p_{D}(a)+p_{D}(b)=0$ by the correspondence property.

\begin{figure}[h]
\centering\includegraphics[width=0.2\textwidth]{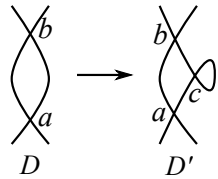}
\caption{Proof of the equality $p(a)+p(b)=0$}\label{fig:r2_parity}
\end{figure}

Note that the statement holds for prisms which give an alternating bigon in the section because we can choose between two under-oversheet structure of the pinch.

2. Let $a$ be a double line in a diagram $D$. Consider the section of the diagram by a transversal plane at some point of $a$. It looks like a crossing (Fig.~\ref{fig:parity_mod2} left). Apply Roseman moves $\mathcal R_3$ and $\mathcal R_5$ so that they give the diagrams as shown in Fig.~\ref{fig:parity_mod2} in the section by the transversal plane. Then we have $p_{D'}(a)+p_{D'}(b)=0$ by the first statement. In the diagram $D''$, $p_{D''}(a)+p_{D''}(c)+p_{D''}(d)=0$ and $p_{D''}(b)+p_{D''}(c)+p_{D''}(d)=0$ by Proposition~\ref{prop:R5_parity}. Hence, $p_{D''}(a)=p_{D''}(b)$. By the correspondence property, $p_{D'}(a)=p_{D'}(b)$. On the other hand, we have $p_{D'}(a)+p_{D'}(b)=0$ by the first statement of the proposition. Therefore, $2p_{D'}(a)=0$ and $2p_D(a)=0$.

\begin{figure}[h]
\centering\includegraphics[width=0.6\textwidth]{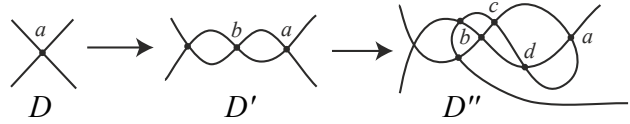}
\caption{Proof of the equality $2p(a)=0$}\label{fig:parity_mod2}
\end{figure}
\end{proof}

\begin{definition}
Parity $p_u$ with coefficients in $A_u$ is called a {\em universal
parity} if for any parity $p$ with coefficients in $A$ there exists
a unique homomorphism of group $\rho: A_u\to A$ such that
$p_D=\rho\circ (p_u)_D$ for any diagram $D$.
\end{definition}

By definition, the universal parity accumulates maximal information on the double lines among all parities.

The universal parity for $1$-dimensional free knots was described in~\cite{IMN1}.
\begin{theorem}
Gaussian parity (with coefficients in $A_u = \Z_2$) is the universal parity of
free knots.
\end{theorem}

In the next section we prove the analogous result for free $2$-knots.

\section{Main result}\label{sec:main_result}

The main result of the paper is the following theorem.

\begin{theorem}\label{thm:main_theorem}
The Gaussian parity (with coefficients in $A_u = \Z_2$) is the universal parity of free $2$-knots.
\end{theorem}

We start the proof with auxiliary definitions and lemmas. Let $p$ be a parity with coefficients in $A$ on diagrams of a virtual knot.

\begin{definition}
Let $x$ and $y$ be points in a diagram $D$ which do not lie on a double line. The {\em potential} between points $x$ and $y$ is the parity of the double line created by the Roseman move $\mathcal R_1$ as shown in the picture below: $\delta_{x,y}=p(a)\in A$.

\begin{figure}[h]
\centering\includegraphics[height=0.12\textwidth]{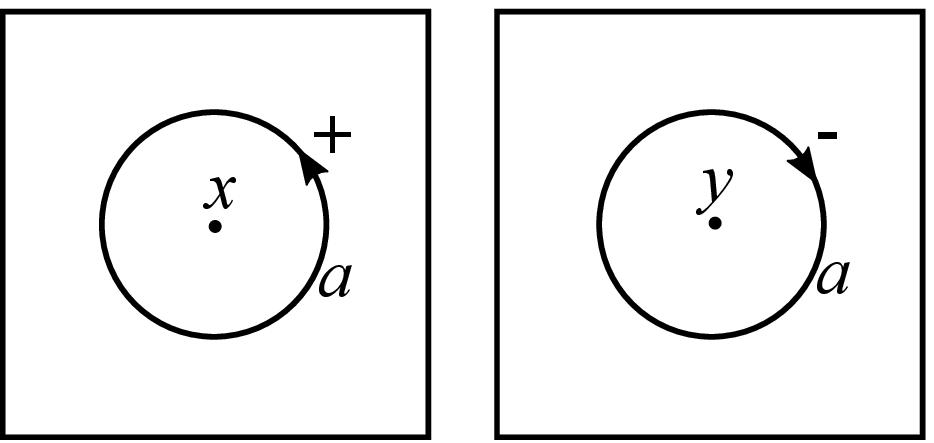} \qquad \includegraphics[height=0.12\textwidth]{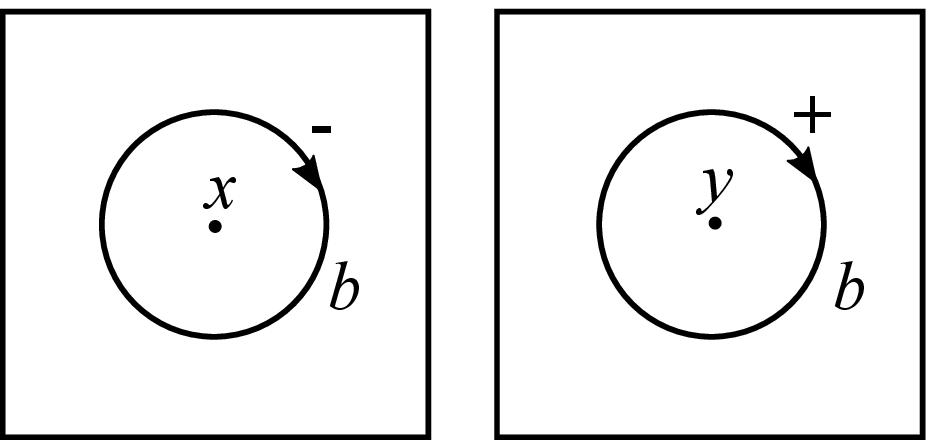} \qquad \includegraphics[height=0.12\textwidth]{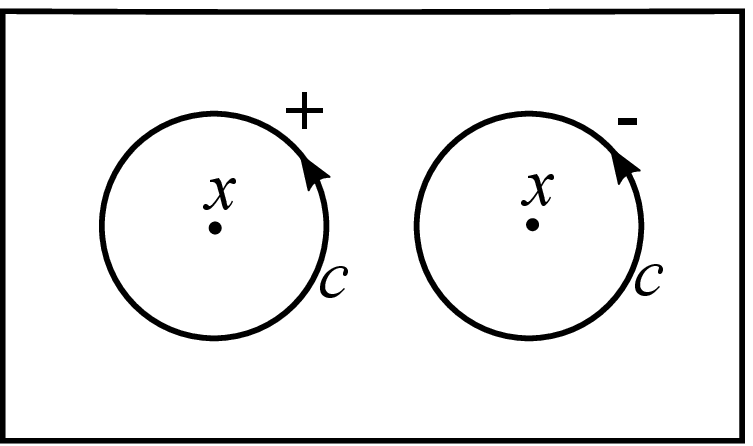}
\caption{Potentials  $\delta_{x,y}=p(a)$,  $\delta_{\bar y, x}=p(b)$ and $\delta_{x}=\delta_{x,\bar  x}=p(c)$}
\end{figure}

Analogously, elements $\delta_{\bar x,y}, \delta_{\bar y,\bar x}$ etc. can be defined. The first index corresponds to the curve with the sign $+$. The bar mark corresponds to the clockwise (negative) orientation of the curve for the first index, and to counterclockwise (positive) orientation of the curve for the second index. Let $\delta_x=\delta_{x,\bar x}$.
\end{definition}

\begin{lemma}\label{lem:delta_properties}
\begin{enumerate}
\item For any regular points $x$ and $y$ we have $\delta_{x,y}=\delta_{y,x}$, $\delta_{x,\bar y}=\delta_{\bar y, x}$ and $\delta_{\bar x,\bar y}=\delta_{\bar y,\bar x}$.
\item For any regular points $x,y,z$ the equality $\delta_{x,z}=\delta_{x,y}+\delta_{y,z}$ holds.
\item If parity $p$ is invariant under virtualization then for any $x$ and $y$ we have $\delta_{x,\bar y}=\delta_{y,\bar x}$, $\delta_{x,y}=\delta_{\bar y,\bar x}$ and $\delta_{x}=\delta_{y}$.
\end{enumerate}
\end{lemma}

\begin{proof}

1. Apply Roseman moves $\mathcal R_1$ to produce double lines for  $\delta_{x,y}$ and $\delta_{y,x}$. Then the double lines form locally a bigonal prism, so $\delta_{x,y}+\delta_{y,x}=0$ by Proposition~\ref{prop:parity_mod2}. Then $\delta_{x,y}=\delta_{y,x}$. The other equalities are proved in the same way.

2. Apply Roseman moves $\mathcal R_1$ for $\delta_{x,z}$, $\delta_{x,y}$ and  $\delta_{y,z}$ so that the section of the diagram by a transversal plane looks as in Fig.~\ref{fig:poten_add}. Then the double lines form a trigonal prism and $\delta_{x,y}+\delta_{y,z}+\delta_{x,z}=0$ by Proposition~\ref{prop:R5_parity}. Since $2\delta_{x,z}=0$, $\delta_{x,z}=\delta_{x,y}+\delta_{y,z}$.

\begin{figure}[h]
\centering\includegraphics[height=0.25\textwidth]{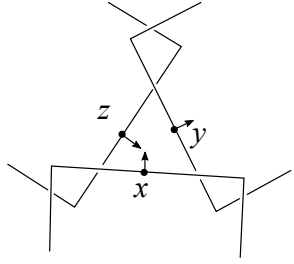}
\caption{Proof of $\delta_{x,y}+\delta_{y,z}+\delta_{x,z}=0$}\label{fig:poten_add}
\end{figure}

3. Since virtualization interchanges the signs of double lines and keeps their orientations, we have $\delta_{x,\bar y}=\delta_{y,\bar x}$ and
$\delta_{x,y}=\delta_{\bar y,\bar x}$.

By the second statement, $\delta_{x,\bar y}=\delta_{x,y}+\delta_{y,\bar y}=\delta_{x,y}+\delta_y$ and $\delta_{y,\bar x}=\delta_{y,x}+\delta_{x,\bar x}=\delta_{y,x}+\delta_x$. But $\delta_{x,y}=\delta_{y,z}$ by the first statement, and $\delta_{x,\bar y}=\delta_{y,\bar x}$ as shown above. Hence, $\delta_{x}=\delta_{y}$.

\end{proof}

When the parity $p$ is invariant under virtualization we denote $\delta=\delta_{x}\in A$ for some regular $x$. The value $\delta$ does not depend on the  point $x$. Note that $2\delta=0$ by Propostion~\ref{prop:parity_mod2}.


\begin{lemma}\label{lem:parity_potential}
Let $x,y,z,u$ be points near a double line $a$ as shown in Fig.~\ref{fig:parity_poten}. Then $p(a)=\delta_{x,y}=\delta_{y,\bar z}=\delta_{\bar z,\bar u}=\delta_{\bar u, x}$.
\end{lemma}

\begin{proof}

Apply a Roseman move $\mathcal R_1$ which creates a double line $b$ with $p(b)=\delta_{x,y}$. Then the lines $a$ and $b$ form locally a bigonal prism, so $p(a)=p(b)=\delta_{x,y}$ by Proposition~\ref{prop:parity_mod2}. The other equalities are proved analogously.

\begin{figure}[h]
\centering\includegraphics[height=0.2\textwidth]{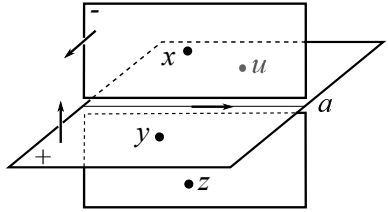}
\caption{Potentials near a double line}\label{fig:parity_poten}
\end{figure}

\end{proof}

\begin{corollary}\label{cor:delta_double_line}
Let the parity $p$ be invariant under virtualization. Then for any regular points $x$ and $z$, or $y$ and $u$ separated by a double line (see Fig.~\ref{fig:parity_poten}) we have $\delta_{x,z}=\delta_{y,u}=\delta$.
\end{corollary}

\begin{proof}
  By Lemma~\ref{lem:parity_potential}, $p(a)=\delta_{x,y}=\delta_{y,\bar z}$. Then $\delta_{x,z}=\delta_{x,y}+\delta_{y,\bar z}+
  \delta_{\bar z,z}=p(a)+p(a)+\delta=\delta$. Analogously, $\delta_{y,u}=\delta_{y,x}+\delta_{x,\bar u}+\delta_{\bar u,u}=p(a)+p(a)+\delta=\delta$.
\end{proof}

Now we can finish the proof of the main theorem
\begin{lemma}
For any parity $p$ with coefficients $A$ on diagrams of free $2$-knots (or virtual $2$-knots up to virtualization) there exists an element $\delta\in A$, such that for any double line $a\in $ $p(a)=p^G(a)\cdot\delta$.
\end{lemma}

\begin{proof}[Proof of Theorem~\ref{thm:main_theorem}]

Let $p$ be a parity with coefficients $A$ on diagrams of free $2$-knots (or virtual $2$-knots up to virtualization). We will show that for the element $\delta\in A$ defined above, for any double line $a\in \mathcal D(D)$ we have $p(a)=p^G(a)\cdot\delta$. Then the homomorphism $\rho\colon \Z_2\to A$ form Gaussian parity to the parity $p$ is $\rho(x)=x\cdot\delta$.

  Let $a\in \mathcal D(D)$ be a double line in a diagram $D$. In the corresponding Gauss diagram the line $a$ is presented by a pair of curves $a_+,a_-$. Choose a path $\gamma$ that connects $a_+$ with $a_-$ such that the tangent vectors $\dot\gamma$, $\dot a_+$ form positively oriented basis at the starting point of $\gamma$, and $\dot\gamma$, $\dot a_-$ are positively oriented at the final point of $\gamma$ (Fig.~\ref{fig:gaussian_2parity}).

  Choose a point $y\in\gamma$ near the start of $\gamma$, and $x\in\gamma$ near the end of $\gamma$. By Lemma~\ref{lem:parity_potential} $p_D(a)=\delta_{x,y}$ (see Fig.~\ref{fig:parity_poten}). For any regular point $z\in\gamma$ denote $f(z)=\delta_{z,y}\in A$. Then $p_D(a)=f(x)$. On the other hand, $f(y)=\delta_{y,y}=0$, and by Corollary~\ref{cor:delta_double_line} while moving along $\gamma$ from $y$ to $x$ the value $f(z)$ changes by $\delta$ whenever the point $z$ crosses a double line. Hence, $\delta_{x,y}=\delta\cdot k$ where $k$ is the number of intersection points of $\gamma$ with double line. Then $k\equiv p^G(a) \mod 2$, thus, $p_D(a)=p^G(a)\cdot\delta$.
\end{proof}

\section{Oriented parity}\label{sec:oriented_parity}

In~\cite{N} a definition of oriented parity was introduced. Oriented parities include parities with coefficients and in general are not subjected to the relation $2p(a)=0$. Here we give a definition for $2$-knots.

\begin{definition}
Let $\mathcal K$ be a virtual $2$-knot. A {\em oriented parity} with coefficients in an abelian group $A$ on the diagrams of a $2$-knot $\mathcal K$ is a family of maps $p_D\colon \mathcal D(D)\to A$, where $\mathcal D(D)$ is the set of (oriented) double lines of an abstract diagram $D$ of the knot $\mathcal K$, which obeys the following properties.
\begin{enumerate}
\item ({\em local correspondence property}) let diagrams $D$ and $D'$ be connected by a Roseman move. If double lines $a\in\mathcal D(D)$ and $a'\in\mathcal D(D')$ contain arcs which do not affected by the move and correspond to each other then $p_D(a)=p_{D'}(a')$.
\item ({\em oriented triple point property}) let double lines $a,b,c\in\mathcal D(D)$ intersect in a triple point of the diagram $D$ (Fig.~\ref{fig:oriented_triple_property}). Then $$\epsilon(a)p_D(a)+\epsilon(b)p_D(b)+\epsilon(b)p_D(c)=0$$
    where $\epsilon(a)=or(a,b)or(a,c)$, $\epsilon(b)=or(b,a)or(b,c)$, $\epsilon(c)=or(c,a)or(c,b)$ and $or(a,b)$ is the orientation of the basis formed by the double lines $a$ and $b$ at the intersection point in the sheet $\gamma$.
\item ({\em pinch property}) the parity the double line $a$ of any pinch is zero: $p_D(a)=0$. A {\em pinch} is the pair of two curves which appears after a Roseman move $\mathcal R_3$.
\end{enumerate}

\begin{figure}[h]
\centering\includegraphics[width=0.8\textwidth]{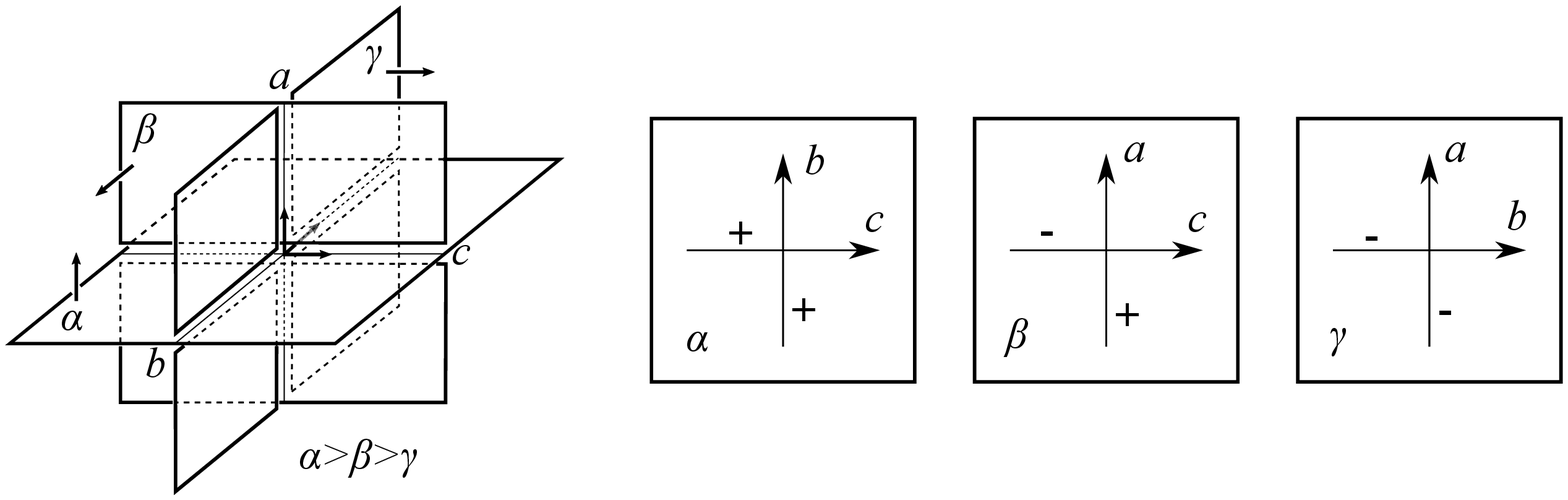}
\caption{Oriented triple point property}\label{fig:oriented_triple_property}
\end{figure}
\end{definition}

\begin{remark}
\begin{enumerate}
\item the parity $p_D(a)$ of a double line $a$ is bound to the orientation of $a$. This means that if one admits a crossing switching operation (for example when $p$ is a parity on flat or free knot) then a crossing switch will change the sign of the parity $p_D(a)$ of the double line the switch is applied to because the move changes the orientation of the double line. Virtualizations (if they are admissible) do not change the parity of double lines because they do not change the orientations. The oriented triple point property is compatible with crossing switch and virtualization operations.

\item if a triple point is geometrical (see Remark~\ref{rem:geometrical_triple_point}) then the oriented triple point property can be written as $$p_D(a)-p_D(b)+p_D(c)=0$$ where $b$ is the intersection of the upmost sheet and the lowest sheet.

    For example, in the Fig.~\ref{fig:oriented_triple_property} we have $or(a,b)=-1$, $or(a,c)=-1$, $or(b,c)=-1$, hence, $\epsilon(a)=(-1)\cdot(-1)=1$, $\epsilon(b)=1\cdot(-1)=-01$, $\epsilon(c)=1\cdot1=1$. Thus, the oriented triple point property is $p_D(a)-p_D(b)+p_D(c)=0$.
\end{enumerate}
\end{remark}

\begin{example}[Index parity]

Let $a$ be a double line in a diagram $D$. Connect the positive curve $a_+$ and the negative curve $a_-$ in the Gauss diagram with a path $\gamma$ such that $\gamma$ and $a_\pm$ form positively oriented bases at the endpoints of $\gamma$. Take the algebraic sum of the intersections of $\gamma$ with double lines: to any intersection point of $\gamma$ with a double line $b$ with the sign $\epsilon$ we assign the number $\epsilon\cdot or(\gamma,b)$ where $or(\gamma,b)$ is the orientation of the basis $(\gamma,b)$ on the sphere. The sum of these numbers is the index parity $ip_D(a)$ of the double line $a$.

\begin{figure}[h]
\centering\includegraphics[width=0.3\textwidth]{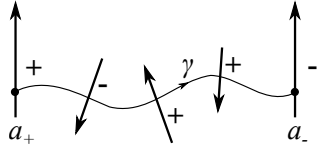}
\caption{Double point index as oriented Gaussian parity}\label{fig:gauss_2parity_or}
\end{figure}

For example, the index parity of the double line in Fig.~\ref{fig:gauss_2parity_or} is $ip_D(a)=(-1)\cdot(-1)+1\cdot 1+1\cdot(-1)=1$.

\begin{figure}[h]
\centering\includegraphics[width=0.7\textwidth]{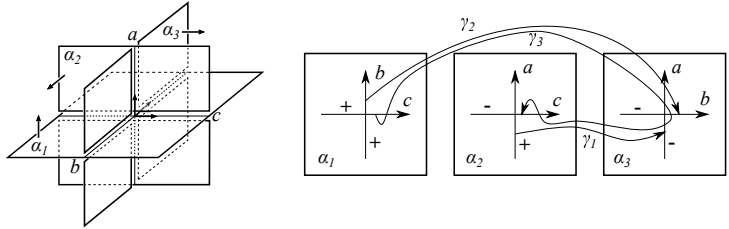}
\caption{Oriented triple point property for the index parity}\label{fig:oriented_triple_property_ex}
\end{figure}

Let us show that index parity obeys the oriented triple point property. Consider three double lines $a,b,c$ which intersects in a triple point (Fig.~\ref{fig:oriented_triple_property_ex}). Connect the preimages of the double line in the Gauss diagram with paths $\gamma_1$, $\gamma_2$ and $\gamma_3$ as shown. Then $\gamma_3$ is homotopic to $\gamma_2\circ \gamma_1^{-1}$. The index parity $ip_D(a)$ is the algebraic sum of intersections on  $\gamma_1$, and $ip_D(b)$ is the algebraic sum of intersections on $\gamma_2$. Then
\begin{multline*}
ip_D(c)=1\cdot(-1)+ip_D(b)+(-1)\cdot 1+(-1)\cdot(-1) - ip_D(a)+(-1)\cdot(-1)=\\ ip_D(b)-ip_D(a).
\end{multline*}

Thus, $ip_D(a)-ip_D(b)+ip_D(c)=0$.

Non geometric triple points can be checked analogously.
\end{example}


\end{document}